\documentclass{amsart}
\usepackage{amsmath, amsthm}
\usepackage{amsfonts,amssymb}
\usepackage{epsfig}
\usepackage{graphicx}
\usepackage{subfigure}
\newcommand{\R}{{\mathbb R}}

\newcommand{\B}{\bigskip}
\newcommand{\m}{\medskip}
\newcommand{\bb}{_\bullet}

\newcommand{\cal}{\mathcal}

\newtheorem{theorem}{Theorem}[section]
\newtheorem{cor}{Corollary}

\newtheorem{prop}[theorem]{Proposition}

\theoremstyle{definition}
\newtheorem{definition}[theorem]{Definition}
\newtheorem{example}[theorem]{Example}
\newtheorem{remark}{Remark}

\numberwithin{equation}{section}

\numberwithin{equation}{section}
\def\sqr#1#2{{\vcenter{\vbox{\hrule height.#2pt
 \hbox{\vrule width.#2pt height#1pt \kern#1pt
 \vrule width.#2pt}
 \hrule height.#2pt}}}}

\title{A Combinatorial Approach to Binary Positional Number Systems} 
\author[A. Vince]{Andrew Vince}
 \address{Department of Mathematics, University
of Florida \\ Gainesville, FL, USA}
\email{avince@ufl.edu }

\begin{document}

\begin{abstract}  Although the representation of the real numbers in terms of a base and
a set of digits has a long history,  new questions arise even in the binary case - digits $0$ and $1$.  
A binary positional number system (binary radix system) with base equal the golden ratio $(1+\sqrt{5})/2$ 
is fairly well known.  The main result of this paper is a construction of infinitely many binary radix systems, each one constructed  combinatorially from a single pair of binary strings. Every binary radix system that satisfies even a minimal set of conditions that would be expected of a positional number system, can be constructed in this way.  
\end{abstract}

\maketitle


\section{Introduction}  \label{secIntro}

The terms {\it positional number system, radix system}, and $\beta$-{\it expansion} that appear in the literature all refer to the representation of real numbers  in terms of a given base or radix $B$ and a given finite set $D$ of digits.  Historically the base is $10$ and the digit set is $\{0,1,2, \dots , 9\}$ or, in the binary case, the base is $2$ and the digit set is $\{0,1\}$.  Alternative choices for the set of digits goes back at least to Cauchy, who suggested the use of negative digits in base $10$, for example $D = \{-4,-3, \dots , 4,5\}$.  The {\it balanced ternary system} is a base $3$ system with digit set $D = \{-1,0,1\}$ discussed by Knuth in \cite{K2}.  In the balanced ternary, every integer, positive or negative, has a representation of the  form $\sum_{n=0}^{N}\omega_n \, 3^n$,
where $\omega_n \in \{-1,0,1\}$ for all $n$.  A well known system with digit set $\{0,1\}$ and base equal the golden ratio $(1+\sqrt{5})/2$ originated with  \cite{B}.  Other early work on the representation of numbers using a non-integer base include those of  R\'enyi \cite{R} and Parry \cite{P}.  Positional number systems  have also been extended to the representation of complex numbers \cite{K1} and, more generally, to the representation of points in $\R^d$ \cite{V}. 
A vast array of additional references on positional number systems can be accessed by searching on the terms `radix system' and `$\beta$-expansion'.

Our intention in this paper is to provide a combinatorial framework for the binary representation of the real numbers. By binary we mean with digit set $\{0,1\}$.  A precise definition of a {\it binary radix system} appears in Section~\ref{sec:Interval} (Definition~\ref{def:radixR}).   Starting from any  pair of infinite strings of $0$'s and $1$'s that satisfy a few combinatorial conditions given in Section~\ref{sec:admissible},  a binary radix system is constructed (Theorem~\ref{thm:radix} and Theorem~\ref{thm:radixR}).  Conversely, any binary radix system  can be obtained by this   construction (Theorem~\ref{thm:converse} and Theorem~\ref{thm:converseR}).  This leads to infinitely many binary radix systems, some of whose properties are investigated in this paper. 

The organization of this paper is as follows.  The notion of an {\it admissible pair} $(\alpha, \beta)$ of strings is defined in Section~\ref{sec:admissible}.  The construction of a binary radix system on the interval $[0,1]$ from
an admissible pair of binary strings is the subject of Section~\ref{sec:Radix}.   Moreover, every binary radix system on $[0,1]$ can be obtained by this construction from some admissible  pair $(\alpha, \beta)$.  Section~\ref{sec:reals} extends these two results of Section~\ref{sec:Radix} to radix systems for the set ${\R}^+$ of non-negative real numbers.  Section~\ref{sec:proofs} contains proofs of results in the previous two sections.  An algorithm is provided in Section~\ref{sec:alg} whose input is a particular binary radix system and a positive real number $x$ and whose output is the decimal expansion of $x$ in that binary radix system.  In addition, it is shown that, not only are there infinitely
many radix systems for ${\R}^+$, but there are infinitely many radix systems for any given base between $1$ and $2$.  For the standard base $2$ radix system there is an associated tiling of the real line, where a tile consists of all those points on the real line with the same integer part.  This is the trivial tiling consisting of unit length intervals.  There is an an analogous tiling for any binary radix system but, in general, the tiles have various lengths, and the tiling is self-replicating but not periodic.  This is the subject of Section~\ref{sec:tiling}.  

In a sense, this paper is a companion to the previous paper \cite{BV}, and we refer to results in that paper in some of the proofs in this paper.  

\section{Binary Radix Systems} \label{sec:Interval}

 Let  $\Omega = \{0,1\}^{\infty}$ denote the set of infinite strings of
the form  $\omega := \omega_0\, \omega_1 \, \omega_2 \cdots$,  where $\omega_n \in \{0,1\}$ for all $n\geq 0$.  
 A line over a finite string denotes infinite repetition, for example $\overline{01} = 010101\cdots$.  The {\it lexicographic order} $\preceq$ on $\Omega$ is the total order defined by  $\sigma \prec\omega$ if $\sigma\neq\omega$ and $\sigma_{k}<\omega_{k}$ where $k$ is the least index such that $\sigma_{k}
\neq\omega_{k}$. 

\begin{definition}  The {\bf shift operator} $S: \Omega \rightarrow \Omega$ is defined by 
$$S(  \omega_0\, \omega_1 \omega_2 \cdots )= \omega_1 \,\omega_{2} \cdots.$$   Let $S^n$ denotes the $n^{th}$ iterate of $S$ for all $n\geq 0$.  A subset $\Gamma \subset \Omega$ is {\bf shift invariant} if $S^n (\Gamma) \subseteq \Gamma$ for all $n\geq 0$.  
\end{definition}

Extend $\Omega$ to a set $\Omega^{\bullet}$ of {\it decimals} by adding a ``decimal point" as follows:
$$ \Omega^{\bullet}  := \left \{ \omega_0\, \omega_1 \cdots \omega_{N}\, \bb  \, \omega_{N+1} \, \omega_{N+2} \cdots \, : \,  \omega_0 \,\omega_1 \, \omega_2 \cdots \in\Omega \right \}.$$
 When no confusion arises, we omit an initial string of zeros (before the decimal point) and/or a terminal string of zeros (after the decimal point), for example $01\bb 1000\cdots = 1\bb 1$.

 The lexicographic order on $\Omega$ can be extended to $ \Omega^{\bullet}$ as follows. If $\omega \in  \Omega^{\bullet}$, let ${\widehat \omega} \in \Omega$ denote the string
 $\omega$ with the decimal point removed.  If 
 $$\begin{aligned} \sigma &=  \sigma_0\, \sigma_1 \cdots \sigma_{N}\, \bb \, \sigma_{N+1} \, \sigma_{N+2} \cdots \quad \text{and} \\ \omega &= 
  \omega_0\, \omega_1 \cdots \omega_{N}\, \bb  \, \omega_{N+1} \, \omega_{N+2} \cdots\end{aligned} $$
(where some leading entries may possibly be $0$'s), then define
$\sigma \prec \omega$ if $\sigma \neq \omega$ and ${\widehat \sigma} \prec {\widehat \omega}$ in the lexicographic order on $\Omega$.  For example $\bb \overline{1} \prec 1\bb \overline{0}$ because 
 $0\bb\overline{1} \prec 1\bb\overline{0}$ because $0\overline{1} \prec 1\overline{0}$.

\begin{definition}  A subset $\Gamma \subset \Omega^{\bullet}$ will be called 
{\bf shift invariant} if, whenever $\omega \in \Gamma$, for any $n\geq 0$, any decimal obtained by placing the decimal point at any posiltion in $S^n({\widehat \omega})$ (introducing zeros if needed), is also in  $\Gamma$.  
\end{definition}

The  conditions in the definition below are meant to be the minimum requirements that one would expect of a positional number system for the set $\R^+$  of non-negative real numbers. 

\begin{definition} \label{def:radixR}
A {\bf binary radix system} for $\R^+$  is a pair $(\Gamma, B)$, where $B>1 $ is a real number called the {\bf base} and $\Gamma \subset \Omega^{\bullet}$ is called the {\bf address space}.   The following conditions must be satisfied:
\begin{enumerate}
\item $\Gamma$ is shift invariant, and
\item the map  $\overset{\bullet}{\pi} \,: \, \Gamma \rightarrow \R^+$  defined by
\begin{equation} \label{eq:radixM2} \overset{\bullet}{\pi} ( \omega_0\, \omega_1 \cdots \omega_{N}\, \bb  \, \omega_{N+1} \, \omega_{N+2} \cdots)=  \sum_{n=-\infty}^{N} \omega_ {N-n}\, B^n \end{equation}
 is bijective and strictly increasing.  The map $\overset{\bullet}{\pi}$ is called the {\bf radix map}.
\end{enumerate} 
\end{definition}

Condition (1) is a consistency requirement.  In the standard base two radix system, for example, if 
$1$ is in the address space, then so must be $$\dots, \, _{\bb}0001, \,_{\bb}001, _{\bb} 01, \,_{\bb} 1, \, 1_{\bb}, 
\, 10_{\bb}, \,100_{\bb}, \, 1000_{\bb}, \, \dots, $$ and if $_{\bb}111\cdots$ is in the address space, then so must be
$$\dots, \, _{\bb}000\overline {1}, \, _{\bb}00\overline {1}, \, _{\bb} 0\overline {1}, \, _{\bb} \overline {1},  \, 1_{\bb}\overline {1}, \, 11_{\bb}\overline {1}, \, 111_{\bb}\overline {1}, \, 1111_{\bb}\overline {1}, \, \dots.$$

A binary radix system for the real numbers in the unit interval $[0,1]$, for technical reasons that
should become clear subsequently, is slightly different.  

\begin{definition} \label{def:radix}
A {\bf binary radix system} for $[0,1]$ is a pair $(\Gamma, B)$, where $B>1 $ is a real number called the {\bf base} and $\Gamma \subset \Omega$ is called the {\bf address space}.  The following conditions must be satisfied, where  $b = 1/B$:
\begin{enumerate}
\item $\Gamma$ is shift invariant, and
\item the map  $\pi \,: \, \Gamma \rightarrow [0,1]$  defined by
\begin{equation} \label{eq:radixM1} \pi ( \omega_0\, \omega_1 \, \omega_{2} \cdots)=  (1-b) \,\sum_{n=0}^{\infty} \omega_ n \, b^n \end{equation}
 is bijective and strictly increasing.  The map $\pi$ is called the {\bf radix map}.
\end{enumerate} 
\end{definition} 

A remark on the factor $1-b$ in Equation~\eqref{eq:radixM1} is in order.  If $B=2$, the standard binary base, then $b = 1/2$ and 
 $$\pi ( \omega_0\, \omega_1 \, \omega_{2} \cdots)=  \sum_{n=1}^{\infty} \omega_{ n-1} \, \left (\frac12 \right )^n,$$
so that $. \omega_0\, \omega_1 \, \omega_{2} \cdots$ is the usual decimal represention of a real number in the interval $[0,1]$.  We hope that, by the end of the paper, the reader will be convinced that the factor $1-b$ is natural. 

\begin{example} [{\it Standard Binary Radix System}] \label{ex:1}
 The standard binary radix system on $[0,1]$ is actually two examples.
Let $B = 2$ and let $\Gamma$ consist of all strings in $\Omega$ except those ending in $0111\cdots$. Then  $(\Gamma, B)$ is a binary radix system for $[0,1]$.  Alternatively, with the same base $B=2$, let $\Gamma$ consist of all strings in $\Omega$ except those ending in $1000\cdots$. Again  $(\Gamma,B)$ is a binary radix system for $[0,1]$.   Note that an address space $\Gamma$ with base $B= 2$ cannot contain both an element that ends in $1\overline 0$ and an element that ends in $0 \overline 1$. This is because, {\it by shift invariance}, if this were so, then both $1\overline 0$ and $0 \overline 1$ must lie in $\Gamma$.  But $\pi( 1\overline 0) = \pi(0 \overline 1)$, contradicting condition (2), that $\pi$ is bijective.  It is precisely to avoid  this type of inconsistency, i.e., that some numbers are represented by a decimal ending in $1\overline 0$ and others are represented by a decimal ending in $0\overline 1$, that we require shift invariance as a condition for a binary radix system.  
\end{example}  

\begin{example} [{\it Golden Ratio Radix System}]  \label{ex:2} 
In the introduction we referred to a binary radix system  $(\Gamma,B)$, where $B$ is the golden ratio $\tau = (1+\sqrt{5})/2$.  For such a radix system on the interval $[0,1]$, the set $\Gamma$ consists of all strings in $\Omega$ that do not contain $011$ or $\overline{01}$ as a substring.  The radix map  $\pi\,:\, \Gamma \rightarrow [0,1]$ is as follows, where
$\overline{\tau} := \frac{1}{\tau} = (\sqrt{5}-1)/2$:
$$\pi( \omega) = (1-\overline{\tau}) \,\sum_{n=0}^{\infty} \omega_ n \, \overline{\tau}^n = 
\sum_{n=2}^{\infty} \omega_ {n-2} \, \overline{\tau}^n.$$  That $\pi$ is bijective and strictly increasing is a special case of general results in Section~\ref{sec:Radix}. 
\end{example} 

\section{Allowable and Admissible Pairs} \label{sec:admissible}

In this section two related terms are defined, {\it allowable} and {\it admissible} pairs of strings.
For $\alpha, \beta \in \Omega$, the notation $$[\alpha, \beta] := \{ \omega \in \Omega \, : \, \alpha \preceq \omega \preceq \beta\}$$ will be used for a closed interval in $\Omega$;  likewise for the open interval $(\alpha, \beta)$ and half open intervals $[\alpha, \beta)$ and $(\alpha, \beta]$. 

\begin{definition} \label{def:exp}  For $\omega \in \Omega$, let $\omega|n = \omega_0 \omega_1 \cdots \omega_n$.
For $\Gamma \subseteq \Omega$, let 
$$\Gamma_n = \{\omega | n \, : \, \omega \in \Gamma\}$$ and let  $|\Gamma_n|$ denotes the cardinality of $\Gamma_n$.
Define the {\bf exponential growth rate} $h(\Gamma)$ of $\Gamma \subseteq \Omega$ by
$$h(\Gamma) = \limsup_{n\rightarrow \infty} \frac1n \, \log | \Gamma_n |.$$ 
\end{definition}

\begin{definition} \label{def:address}
The {\bf address spaces} associated with a pair $(\alpha,\beta)$ of strings in $\Omega$ are defined by
$$\begin{aligned}\Omega_{(\alpha, \beta,-)} &:= \{\omega\in \Omega \, : \, S^n\omega \notin (\alpha,\beta] \; \text{for all} \; n\geq 0\}, \\ 
\Omega_{(\alpha, \beta,+)} &:= \{\omega\in \Omega \, : \,  S^n\omega \notin [\alpha,\beta)\; \text{for all} \; n\geq 0\}, \\ \Omega_{(\alpha, \beta)} &:=\Omega_{(\alpha,\beta,-)}\cup \Omega_{(\alpha,\beta,+)}.\end{aligned}$$
\end{definition}

\begin{definition} \label{def:allowable}
Call a pair  $(\alpha,\beta)$ of elements of $\Omega$ {\bf allowable}
if $\alpha$ and $\beta$ satisfy the following three conditions.
\begin{enumerate}
\item $\alpha_0 = 0,\; \alpha_1 =  1 \quad$ and $\quad \beta_0 = 1, \; \beta_1 = 0$,
\item $S^n\alpha \notin (\alpha,\beta]\quad$ and $\quad S^n\beta \notin [\alpha,\beta)\quad $ for all $n\geq 0$,
\item $h(\Omega_{(\alpha, \beta)}) > 0$.
\end{enumerate}
\end{definition}

\noindent There exist pairs that satisfy conditions (1) and (2), but not (3);  see \cite{BV} for examples.

\begin{example} \label{ex:11}  If $\alpha = 0\overline{1} = 0\, 1 \,1 \,1  \cdots$ and $\beta =  1\overline{0}= 1\,0\,0\,0\cdots$, then $(\alpha, \beta)$ is an admissible pair.   The address space $\Omega_{(\alpha, \beta,-)}$ consists of all strings that do not end in $1000\cdots$, and  $\Omega_{(\alpha, \beta,+)}$ consists of all strings that
do not end in $0111\cdots$.  Notice that    $\Omega_{(\alpha, \beta,-)}$ and  $\Omega_{(\alpha, \beta,+)}$ are exactly the two possible address spaces for the standard binary radix systems of Example~\ref{ex:1}.
\end{example}

\begin{example} \label{ex:22}  If $\alpha =\overline{01}$ and $\beta = 1 \overline 0$, then $(\alpha, \beta)$ is an
admissible pair.   Notice that  $\Omega_{(\alpha,\beta,-)}$ is exactly the address space of the
golden ratio based binary radix system of Example~\ref{ex:2}.  The address space   $\Omega_{(\alpha,\beta,+)}$
consists of all strings not containing $011$ or $1\overline 0$  as substrings, and   $\Omega_{(\alpha,\beta)}$
consists of all strings not containing $011$  as a substring.
 \end{example}

\begin{definition} \label{def:periodic}
 A string $\omega \in \Omega$ will be called  {\bf periodic} if $\omega = \overline{s}$ for some finite string $s$.  A string $\omega \in \Omega$ will be called {\bf eventually periodic} if there is an $m \geq 0$ such that $S^m\omega$ is periodic. 
\end{definition}

The fact that the address space $\Omega_{(\alpha, \beta)}$ in Example~\ref{ex:22} can be characterized by a set of ``forbidden" finite substrings ($011$ in that case) is not a coincidence, as explained in the next proposition.  

\begin{prop} \label{prop:A}  If an admissible pair  $(\alpha, \beta)$ is such that $\alpha$ and $\beta$ are periodic, then there is a finite set $T$ of finite strings such that  $\Omega_{(\alpha, \beta)}$ is the set of $\omega \in \Omega$ such that $s$ is not a substring of $\omega$ for any $s\in T$. 
\end{prop} 

More instructive than giving a formal proof of the proposition is to give an example that makes clear how a proof would proceed. Referring to 
Definition~\ref{def:allowable} of allowable pair,  if $\alpha = \overline{01101}$ and $\beta = \overline{100}$, then the forbidden set is $T = \{ 0111, 011011, 1000\}$.  
The string $0111$ is not allowed because $0111\gamma \succ \alpha$ for any string $\gamma$, and $011011$ is not allowed because  $011011\gamma \succ \alpha$ for any string $\gamma$.  The string $000$ is not allowed because
 $1000\gamma \prec \beta$ for any string $\gamma$.  Because of the periodicity, no additional forbidden strings are
required. 

Note that  Proposition~\ref{prop:A}  is false if, instead of periodic, the hypothesis assumes only eventually periodic.  For example, if $\alpha = 011\overline{01}$ and
$\beta = 1\overline0$, then we would have to forbid an infinite set $0111, 011011, 01101011, 0110101011, \dots$.

A proof of the following theorem appears in \cite{BV}, a main result in that paper. 

\begin{theorem} \label{thm:BV} If $(\alpha,\beta)$ is allowable, then the equation 
 \begin{equation} \label{eq:series} \sum_{n=0}^{\infty} \alpha_n \, x^n = \sum_{n=0}^{\infty} \beta_n \, x^n
\end{equation}
 has a solution in the interval $[\frac12,1)$.  There is no solution in the interval $(0,\frac12)$.  
\end{theorem}

\begin{definition}  \label{def:smallest}  For an allowable pair $(\alpha,\beta)$, denote the smallest real solution of 
equation~(\ref{eq:series}) in the interval $(0,1)$ by $r := r(\alpha,\beta)$.  
\end{definition}

For finite binary strings  $a = \epsilon_0 \epsilon_1\cdots \epsilon_{m-1}$ and $b = \delta_0 \delta_1\cdots \delta_{n-1}$, let
 $$A(x) = \sum_{j=0}^{m-1} \epsilon_ j x^j  \qquad\qquad \text{and}  \qquad\qquad
B(x) = \sum_{j=0}^{n-1} \delta_j x^j,$$ 
and define the polynomial
$$p_{(a,b)}(x) :=  (1-x^m)\, B(x)- (1-x^n)\, A(x).$$
 For $r\in (0,1)$, call finite binary strings $a$ and $b$ $\mathbf r$-{\bf equivalent} if  $r$ is a root of $p_{(a,b)}(x)$.
  It is easily checked that $r$-equivalence is indeed an 
equivalence relation on the set of finite binary strings. 

\begin{example}  \label{ex:equivalent} 
The strings $01$ and $100$ are $r$-equivalent, where $r$ is the real root of $x^3+x^2  - 1$.  In particular 
$p_{(01,100)}(x) = (1-x)(1-x-x^3)$. \m
The three strings:
$$\begin{aligned}  a &= 011110 \\ b &= 100111 \\ c &= 101010 \end{aligned}$$
are pairwise $r$-equivalent, where $r$ is the real root of $x^3+x-1$.  
In particular $p_{ab}(x) = (1-x^6)(1-x^2)(1-x-x^3)$, $p_{ac}(c) = (1-x^6)(1-x-x^3)$, and $p_{bc}(x) = (1-x^6) x^2 (1-x-x^3)$.
\end{example}

An $r$-equivalence class $C$ is closed under concatenation and is therefore a submonoid of the free monoid $\{0,1\}^*$ consisting of all finite strings with alphabet $\{0,1\}$ and with concatenation as the operation.  It is easy to check that if $a, ab \in  C$ then $b \in C$, and  if $b, ab \in  C$ then $a \in C$, i.e., $C$ is {\it left and right unitary}.  In  particular, $C$ is itself a free monoid with a unique set $G_C$ of free generators.  

\begin{definition} \label{def:admissible}
Consider a pair $(\alpha, \beta)$ of binary strings
$$\alpha = a_0 a_1 a_2 \cdots  \qquad \qquad \beta = b_0 b_1 b_2 \cdots,$$
where $S:= \{a_0,a_1,\dots, b_0,b_1, \dots\}$  is a finite subset (with repetition) of the set 
$G_C$ of generators of an $r$-equivalence class $C$. 
Then $(\alpha, \beta)$ will be called $\mathbf r$-{\bf bad} unless $\alpha = aaa\cdots, \, \beta = bbb \cdots$.
A pair of  strings that is not $r$-bad  will be called $\mathbf r$-{\bf good}.   An $r(\alpha,\beta)$-good pair $(\alpha,\beta)$ of allowable strings will be called {\bf admissible}.
\end{definition}

\begin{remark}
It follows immediatly from the fact that $p{(a,b)}(x)$ is a polynomial that, if $(\alpha, \beta)$ is allowable and $r(\alpha,\beta)$ is not an algebraic number, then $r(\alpha,\beta)$ is admissible.
\end{remark}

\begin{remark} It is proved in \cite{BV} that, if $(\alpha, \beta)$ is allowable and $r$-bad for some $r\in (0,1)$, then 
$\sum_{n=0}^{\infty} \alpha_n \, r^n = \sum_{n=0}^{\infty} \beta_n  \, r^n$. 
This does not mean, however, that $r=r(\alpha,\beta)$, because $r$ may not be the {\it smallest} solution of
the equation $\sum_{n=0}^{\infty} \alpha_n \, x^n = \sum_{n=0}^{\infty} \beta_n  \, x^n$ in the interval
$(0,1)$.  On the other hand, we have no example where it is not the smallest. 
\end{remark}

\begin{example}  [{\it Allowable but Not Admissible Pairs}] \label{ex:bad}
The allowable pair
$$\alpha = 011\overline{100} \qquad \qquad \qquad \beta = 100\overline{011}$$
is bad.  Solving equation~\eqref{eq:series}, we find that the number $r:= r(\alpha,\beta)$ is the positive real root of the polynomial $x^2+x-1$, which is 
approximately $0.6180$. It is easy to check that $100$ and $011$ are $r$-equivalent. Therefore
$(\alpha,\beta)$ is bad, hence not admissible.
\m

The allowable pair
$$\alpha = 01\overline{100} \qquad \qquad \qquad \beta = 100\overline{01}$$
is also bad.  Solving equatiion~\eqref{eq:series}, we find that the number $r:= r(\alpha,\beta)$ is 
the real root of the polynomial $x^3+x^2-1$, which is 
approximately $0.52818$.  Again it is easy to check that $100$ and $01$ are $r$-equivalent.
Therefore $(\alpha,\beta)$ is not admissible.
\end{example}

\begin{example}  [{\it Another Allowable but Not Admissible Pair}]  \label{ex:example} Another bad allowable pair is
$$\alpha = accc\cdots  \qquad \qquad \qquad \beta = bababa \cdots,$$
where $a,b$ and $c$ are the finite strings given in Example~\ref{ex:equivalent}. In this example $a,b$ and
$c$ are pairwise $r$-equivalent, where $r$ is the real root of $x^3+x-1$.   It is straightforward
to check that the (not roots of unity) solutions to equation~\eqref{eq:series} satisfy $(x^3+x-1)(x^{12}+x^6+x^2-1)=0$.  Since the only positive real root of $x^{12}+x^6+x^2-1$  is approximately $0.8062$,
while the real root of $x^3+x-1$ is smaller,  approximately $0.6823$, we have $r(\alpha,\beta) = r$.  
\end{example} 

\section{Radix Systems on $[0,1]$ from Admissible Pairs} \label{sec:Radix}

In the previous section,  address spaces 
$\Omega_{(\alpha, \beta,-)}$  and $\Omega_{(\alpha, \beta,+)}$ were defined  for every pair  $(\alpha,\beta)$ of
binary strings.  This section explains how these two address spaces, in the case where  $(\alpha,\beta)$ is
admissible, become the address spaces of two corresponding binary radix systems
for $[0,1]$.  Moverover, the converse also holds.  Every binary radix system for $[0,1]$ can be constructed in this way.  
The proofs of Theorem~\ref{thm:radix} and Theorem~\ref{thm:converse} below are omitted because they are essentially the same as the somewhat more difficult proofs of the analogous Theorems~\ref{thm:radixR} and \ref{thm:converseR} for the reals in Section~\ref{sec:reals}.

\begin{theorem} \label{thm:radix}
If  $(\alpha,\beta)$ is an admissible pair and $B_{(\alpha,\beta)} = 1/r(\alpha,\beta)$, then \linebreak $(\Omega_{(\alpha, \beta, -)} , B_{(\alpha,\beta)})$  and $(\Omega_{(\alpha, \beta, +)},  B_{(\alpha,\beta,-)})$ are binary radix systems for $[0,1]$.
\end{theorem}

\begin{definition} To simplify notation, let
$$\begin{aligned} R_{(\alpha, \beta, +)} := (\Omega_{(\alpha, \beta, +)},  B_{(\alpha,\beta)}) \\
 R_{(\alpha, \beta, -)} := (\Omega_{(\alpha, \beta, -)},  B_{(\alpha,\beta)}) 
\end{aligned}$$ 
denote the radix systems constructed from an admissible  pair  $(\alpha,\beta)$ as in Theorem~\ref{thm:radix}.
 Call  $R_{(\alpha, \beta, +)}$ and $ R_{(\alpha, \beta, -)}$  the $(\alpha,\beta)$-{\bf radix systems}.  
\end{definition}

As stated in the next theorem, all binary radix systems are $(\alpha,\beta)$-radix systems for some admissible $(\alpha,\beta)$.

\begin{theorem} \label{thm:converse}
 For every binary radix system $(\Gamma, B)$ for $[0,1]$, there is an admissible  pair $(\alpha,\beta)$ such that  either $ (\Gamma, B) = R_{(\alpha, \beta, -)}$ or  $(\Gamma, B) = R_{(\alpha, \beta, +)}$.
\end{theorem} 

\begin{example}  [{\it Standard Binary Radix System}]\label{ex:111} 
Continuing from  Example~\ref{ex:1} and Example~\ref{ex:11}, consider the admissible pair $(0\overline{1},1\overline{0})$. The number $r:=r(\alpha,\beta)$  in Theorem~\ref{thm:radix} is the smallest solution to the equation $\sum_{n=1}^{\infty} x^n = 1$, which, because the left hand side is a geometric series, reduces to $2x=1$.  So $r =1/2$ and the base is $B_{(\alpha,\beta)} = 1/r = 2$.  
The radix systems  $R_{(0\overline{1},1\overline{0}, +)}$ and $ R_{(0\overline{1},1\overline{0}, -)}$ are the standard binary radix systems.  
\end{example}

\begin{example}  [{\it Golden Ratio Radix System}] \label{ex:222} 
 Continuing from  Example~\ref{ex:2} and Example~\ref{ex:22}, consider the admissible pair $(\overline{01},1\overline{0})$.  The number $r:=r(\alpha,\beta)$ in Theorem~\ref{thm:radix} is the smallest solution to the equation $\sum_{n=0}^{\infty} x^{2n+1} = 1$, which reduces to $x^2+x-1 = 0$.  So $r = (\sqrt{5}-1)/2$, and the base $B_{(\alpha,\beta)}= 1/r =(1+\sqrt{5})/2 $ is the golden ratio. The radix systems  $R_{(\overline{01},1\overline{0}, +)}$ and $ R_{(\overline{01},1\overline{0}, -)}$ are the golden ratio radix systems.  
\end{example}

\begin{example} \label{ex:3}  Consider the $(\alpha,\beta)$-radix systems $R_{(\alpha_1,\beta_1,\pm)},\, R_{(\alpha_2,\beta_2,\pm)}$, and \linebreak $R_{(\alpha_3,\beta_3,\pm)}$, where  the admissible  pairs are: 
 $$\begin{aligned}\alpha_1 &= \overline{01000}\\ \beta_1 &= 1\overline{0} \end{aligned} \hskip 5mm 
\begin{aligned}\alpha_2 &= \overline{011}\\ \beta_2&= \overline{10} \end{aligned} \hskip 5mm 
\begin{aligned}\alpha_3 &= \overline{01}\\ \beta_3 &= \overline{100}.\end{aligned}$$
In all three cases the number $r:=r(\alpha,\beta)$ in Theorem~\ref{thm:BV} is the solution to the equation $x^3+x^2-1=0$ in the interval $(0,1)$; approximately  $r  \approx 0.7549$.   Therefore the base $B_{(\alpha,\beta)}\approx 1.3247$ and the radix maps are the same in all three cases.  The three address spaces, however, are pairwise distinct.  The address space $\Omega_{(\alpha_1,\beta_1)}$ consists of all strings that do not contain the substrings $11, 101, 1001$ or $10001$.   The address space $\Omega_{(\alpha_2,\beta_2)}$ consists of all strings that do not contain the substrings $100$ or $111$.  The address space $\Omega_{(\alpha_3,\beta_3)}$ consists of all strings that do not contain the substrings  $11$ or $1000$.  
\end{example}

\section{Binary Radix Systems for the Non-Negative Reals} \label{sec:reals}

In this section, the method of the previous section for constructing radix systems for the interval $[0,1]$ is extended
in order to construct radix systems for the set $\R^+$ of non-negative real numbers. 

\begin{definition} \label{def:addressE}
Extend the definition of address spaces $\Omega_{(\alpha, \beta,-)},\Omega_{(\alpha, \beta,+)},\Omega_{(\alpha, \beta)}$ to address spaces $\Omega^{\bullet}_{(\alpha, \beta,-)},  \Omega^{\bullet}_{(\alpha, \beta,+)}, \Omega^{\bullet}_{(\alpha, \beta)}$  of  decimals as follows:
$$\begin{aligned}& \Omega^0_{(\alpha, \beta,-)} := \{\omega\in \Omega_{(\alpha, \beta,-)} \, : \,0\omega \preceq \alpha\}, \\ 
&\Omega^0_{(\alpha, \beta,+)} := \{\omega\in \Omega_{(\alpha, \beta,+)} \, : \,   0\omega \prec \alpha\}, \\
&\Omega^0_{(\alpha, \beta)} :=\Omega^0_{(\alpha,\beta,-)}\cup \Omega^0_{(\alpha,\beta,+)}.\end{aligned}
\qquad \qquad \qquad
\begin{aligned}
  &\Omega^{\bullet}_{(\alpha, \beta,-)}  := [\, \Omega^0{(\alpha, \beta,-)} \,]^{\bullet},
 \\ &\Omega^{\bullet}_{(\alpha, \beta,+)} := [\, \Omega^0{(\alpha, \beta,+)} \,]^{\bullet}, \\
& \Omega^{\bullet}_{(\alpha, \beta)} := [\, \Omega^0{(\alpha, \beta)} \,]^{\bullet}.
\end{aligned}$$
\end{definition}
\noindent The reason for introducing the spaces $\Omega^0$ with the added  condition $0\omega \prec \alpha$ is to
insure that the spaces $\Omega^{\bullet}$ are shift invariant; see Proposition~\ref{prop:Bshift} below.  

\begin{theorem} \label{thm:radixR}
 If  $(\alpha,\beta)$ is an admissible pair, $b$ the smallest solution to
equation~(\ref{eq:series}) in the interval $(0,1)$, and $ B_{(\alpha,\beta)} = 1/b$, then $(\Omega^{\bullet}_{(\alpha, \beta, -)} , B_{(\alpha,\beta)})$  and $(\Omega^{\bullet}_{(\alpha, \beta, +)},  B_{(\alpha,\beta)})$ are binary radix systems for $\R^+$.
\end{theorem}

\begin{definition} To simplify notation, let
$$\begin{aligned} R^{\bullet}_{(\alpha, \beta, +)} := (\Omega^{\bullet}_{(\alpha, \beta, +)},  B_{(\alpha,\beta)}) \\
 R^{\bullet}_{(\alpha, \beta, -)} := (\Omega^{\bullet}_{(\alpha, \beta, -)},  B_{(\alpha,\beta)}) 
\end{aligned}$$ 
denote the radix systems constructed from an admissible  pair  $(\alpha,\beta)$ as in Theorem~\ref{thm:radixR}.
 Call  $R^{\bullet}_{(\alpha, \beta, +)}$ and $ R^{\bullet}_{(\alpha, \beta, -)}$  the $(\alpha,\beta)$-{\bf radix systems}.  
\end{definition}

 As stated in the next theorem, all binary radix systems are $(\alpha,\beta)$- radix systems for some admissible $(\alpha,\beta)$. 

\begin{theorem} \label{thm:converseR}
 For every binary radix system $(\Gamma, B)$ for $\R^+$, there is an admissible  pair $(\alpha,\beta)$ such that  either $ (\Gamma, B) = R^{\bullet}_{(\alpha, \beta, -)}$ or  $(\Gamma, B) = R^{\bullet}_{(\alpha, \beta, +)}$.
\end{theorem} 

\begin{example} [{\it Standard Binary Radix System}]  This is a continuation of  Examples~\ref{ex:1},\;\ref{ex:11}, and \ref{ex:111} on the standard binary radix system.  For the admissible pair $(0\overline 1, 1\overline 0)$,
in Definition~\ref{def:addressE}, the space $\Omega^0_{(\alpha, \beta, \pm)} = \Omega_{(\alpha, \beta, \pm)}$.
Therefore $\Omega^{\bullet}_{(\alpha, \beta, +)}$ consists of all decimals that do not end in $0\overline 1$ and 
 $\Omega^{\bullet}_{(\alpha, \beta, -)}$ consists of all decimals that do not end in $1\overline 0$.   
\end{example}

\begin{example} [{\it Golden Ratio Radix System}] \label{ex:2222}  This is a continuation of  Examples~\ref{ex:2},\;\ref{ex:22}, and \ref{ex:222} on the golden ratio radix system.  The relavant admissible pair is $(\overline{01},1\overline 0)$.  Note that  $\Omega^0_{(\alpha, \beta)} \neq \Omega_{(\alpha, \beta)}$ since $11\overline 0 \in \Omega_{(\alpha, \beta)} \setminus \, \Omega^0_{(\alpha, \beta)}$.  This is because $011\overline 0 \succ \alpha$.  In partular,  $11\bb \notin  \Omega^{\bullet}_{(\alpha,\beta)}$.  Indeed, it is an easy exercise to
show that $\omega \in \Omega^{\bullet}_{(\alpha, \beta)}$ if and only if  $\omega$ does not
contain $11$ as a substring.
\end{example}

The remainder of this section contains the proof of Theorems~\ref{thm:radixR}, beginning with the proof of shift invariance. The proof of Theorem~\ref{thm:converseR} appears in the next section.

\begin{prop} \label{prop:Bshift} 
Given an admissible pair $(\alpha,\beta)$, the address spaces $\Omega^{\bullet}_{(\alpha, \beta,-)}$,\linebreak  $\Omega^{\bullet}_{(\alpha, \beta,+)}$, and $\Omega^{\bullet}_{(\alpha, \beta)}$ are shift invariant. 
\end{prop} 

\begin{proof} The proposition follows in a straightforward way from the definition of the address spaces, except for one detail.  Consider the case of $\Omega^{\bullet}_{(\alpha, \beta,-)}$ (the case $\Omega^{\bullet}_{(\alpha, \beta,+)}$ is similar).   Again, for $\omega \in  \Omega^{\bullet}$, let ${\widehat \omega} \in \Omega$ denote the string
 $\omega$ with the decimal point removed.  
 If $\omega := \omega_0 \cdots \omega_N {\bb} \omega_{N+1}\omega_{N+2} \cdots \in \Omega^{\bullet}_{(\alpha, \beta,-)}$, then for  $\Omega^{\bullet}_{(\alpha, \beta,-)}$ to be shift invariant it is necessary that $\bb 0\, \omega_0 \omega_1 \omega_2 \cdots \in \Omega^{\bullet}_{(\alpha, \beta,-)}$.  Therefore it is necessary, not only that $S^n{\widehat \omega} \notin (\alpha,\beta]$ for all $n\geq 0$, but also that
$0{\widehat \omega}  \notin (\alpha,\beta]$, i.e., $0{\widehat \omega} \preceq \alpha$.  But this is exactly 
the condition in the defintion of $\Omega^0_{(\alpha, \beta,-)}$.
\end{proof}

\begin{proof} (Theorem~\ref{thm:radixR})
The shift invariance of  $\Omega^{\bullet}_{(\alpha, \beta, \pm)}$, which is condition 1 in  Definition~\ref{def:radix}, is Proposition~\ref{prop:Bshift}.

We next show that  the radix map $\pi  \, : \,   \Omega_{(\alpha, \beta, -)} \rightarrow [0,1]$ in Equation~\eqref{eq:radixM1} is strictly increasing and bijective.   It is  proved in \cite[Lemma 3.10]{BV}  that
$\pi$ is increasing, and in \cite[Proposition 3.2]{BV} that $\pi$ is continuous with respect to the following metric
on $\Omega$:
$$d(\omega,\sigma)=\begin{cases} 2^{-k} \;\; &\text{if} \; \omega\neq \sigma \\
0  \;\; &\text{if} \;  \omega = \sigma, \end{cases}$$ 
where $k$ is the least index such that $\omega_{k}\neq \sigma_{k}$.  To show that $\pi$ is strictly
increasing, let $\sigma,\omega\in \Omega_{(\alpha,\beta,-)}$ with $\sigma \prec \omega$.  Without loss
of generality (by taking a shift) it may be assumed that $\sigma_0 = 0$ and $\omega_0 = 1$. Hence
 $\sigma \preceq \alpha \prec \beta \prec \omega$.  With notation as in Definition~\ref{def:exp}, let $n$ be such that $\beta|(n-1) = \omega|(n-1)$, but $\beta_n = 0, \, \omega_n = 1$.  Since $\alpha \prec   S^n \omega$, we have
$\pi(\alpha) \leq \pi(S^n \omega)$.  Since $S^n\beta \prec \alpha$, it is shown in \cite[Lemma 4.6]{BV} that
$\pi(S^n\beta) < \pi(\alpha)$.  Therefore $\pi(S^n\beta) < \pi(\alpha) \leq \pi(S^n \omega)$ and
therefore $\pi(\sigma) \leq \pi(\beta) < \pi(\omega)$.  That $\pi$ is surjective follows from the 
continuity of $\pi$ and the fact that $\pi(\overline 0) = 0, \, \pi(\overline 1) = 1$.  

It now suffices to prove that the radix map $\overset {\bullet}{\pi}\, : \,  \Omega^{\bullet}_{(\alpha, \beta, \pm)} \rightarrow \R^+$  in Equation~\eqref{eq:radixM2}   is strictly  increasing and bijective. We will prove it for $\Omega^{\bullet}_{(\alpha, \beta, +)}$; the proof for   $\Omega^{\bullet}_{(\alpha, \beta, -)}$ is the same. Abbreviate $B:= B_{(\alpha, \beta)}$ and let $b=1/B$.  Let $p = (1-b) \sum_{n=0}^{\infty} \alpha_n \, b^n = (1-b) \sum_{n=0}^{\infty} \beta_n \, b^n$.  Assume that $\omega \in \Omega_{(\alpha, \beta,+)}$.   Since  $\pi$ is strictly increasing, we have $0 \,\omega \prec \alpha$ 
if and only if  $b \pi(\omega) = \pi (0 \omega) < \pi(\alpha) = p$ if and only if $\pi(\omega) < Bp$.  
Therefore  $\pi \, : \,  \Omega^0_{(\alpha, \beta, +)}\rightarrow [0,Bp]$ is strictly increasing and bijective.
Now let $\Omega^{\bullet}_N$ denote the set of  elements of  $\Omega^{\bullet}_{(\alpha, \beta, +)}$ of the form   $\omega := \omega_0  \omega_1 \omega_2 \cdots \omega_{N-1} {\bb} \omega_N \omega_{N+1} \cdots$, where  ${\widehat \omega}:= \omega_0  \omega_1 \omega_2 \cdots \in  \Omega^0_{(\alpha, \beta,+)}$.  The relationship between the radix map $\pi$ and the radix map $\overset{\bullet}{\pi}$ is  given by $\overset{\bullet}{\pi}(\omega) = \frac{B^{N}}{1-b} \, \pi({\widehat \omega})$, and therefore
$\overset{\bullet}{\pi}\, : \, {\Omega^{\bullet}_N} \rightarrow [\, 0, \frac{B^{N+1}p}{b-1}\,] $ is  a  bijection and strictly increasing. Since the sequence $\Omega^{\bullet}_0 \subset \Omega^{\bullet}_1\subset \Omega^{\bullet}_2\subset \cdots$ of sets is nested, and since
$\bigcup_{N=0}^{\infty} \Omega^{\bullet}_N =  \Omega^{\bullet}_{(\alpha, \beta, +)}$, the proof  is complete.
\end{proof}

\section{Algorithm for Determining the Address} \label{sec:alg} 

Given a binary radix system for $\R^+$, the radix map $\overset{\bullet}{\pi}$ assigns a non-negative real number to each decimal in the address space.   In this section an algorithm is provided for converting in the opposite direction.  Given a non-negative real number $x$, the algorithm determines its decimal representation in the binary radix system. More precisely, if  $R^{\bullet}_{(\alpha,\beta,\pm) }= (\Omega^{\bullet}_{(\alpha, \beta, \pm)}, B_{(\alpha,\beta)})$ has radix map $\overset{\bullet}{\pi} := \overset{\bullet}{\pi}_{(\alpha, \beta)}$ and $x\in \R^+$, then the algorithm finds $\sigma \in \Omega^{\bullet}_{(\alpha, \beta, -)}$ and $\omega \in\Omega^{\bullet}_{(\alpha, \beta, +)}$ such that $\overset{\bullet}{\pi}(\sigma)  = \overset{\bullet}{\pi}(\omega) = x$.  \m

We begin by defining a certain family of functions and introduce notation for the
itineraries of points of this family of functions.  
Given $B$ such that $1<B\leq 2$ and $p$ such that $1-b \leq p \leq b$ where $b=1/B$, 
consider the two functions  $f_{(B,p, \pm)} \, : [0,1] \rightarrow [0,1]$ defined by
\begin{equation} \label{eq:f-}
f_{(B,p, -)}(x) := \begin{cases} Bx \quad  &\text{if $0\leq x \leq p$}  \\  
Bx + (1-B) \quad &\text{if $x > p$}. \end{cases}
\end{equation}
and
\begin{equation} \label{eq:f+}
f_{(B,p, +)}(x) := \begin{cases} Bx \quad  &\text{if $0 \leq x < p$}  \\  
Bx + (1-B) \quad &\text{if $x \geq p$}, \end{cases}
\end{equation}
The facts that $1 < B \leq 2$ and $1-b \leq p \leq b$ guarantee that $f_{(B,p, \pm)}$ has the form shown
in Figure 1.  For a function $f$, the $n^{th}$ iterate, i.e.  $f$ composed with itself $n$ times, is denoted $f^n$.
\begin{figure}[htb] \label{fig:1}
\begin{center} 
\includegraphics[width=2.5in, keepaspectratio]{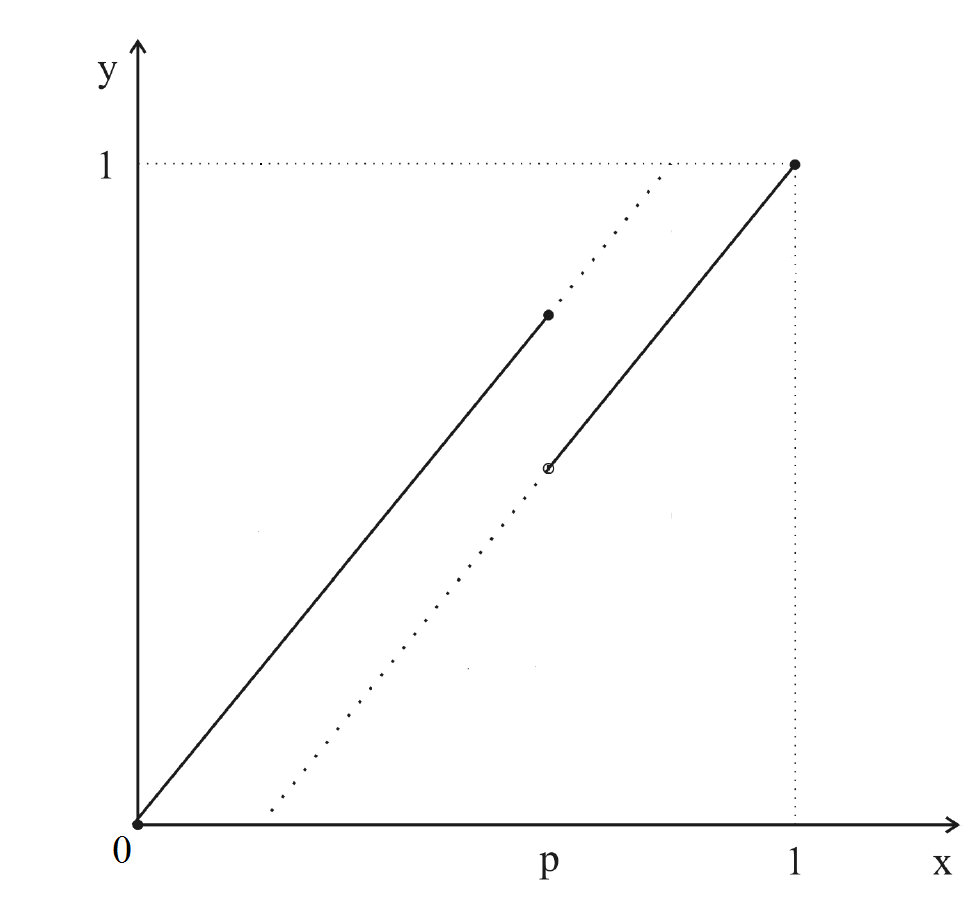}
\caption{The function $f_{(B,p, \pm)}$.}
\end{center}
\end{figure}

\begin{definition} \label{def:itinerary}
Define the two  {\bf itinerary maps} $\tau_{(B,p,\pm)} \, : [0,1] \rightarrow \Omega$  by 
$\tau_{(B,p,-)} = \omega_0\omega_1\omega_2 \cdots$ and    $\tau_{(B,p,+)} = 
\sigma_0\sigma_1\sigma_2 \cdots$, 
where
\begin{equation*}
\omega_k = \begin{cases} 0 \quad \text{if} \;\;\;f_{(B,p, -)}^k(y) \leq p \\ 1  \quad \text{if} \; \;\;f_{(B,p, -)}^k(y) > p,
\end{cases}  \qquad \quad \text{and} \quad \qquad 
\sigma_k = \begin{cases} 0 \quad \text{if} \; \;\; f_{(B,p, +)}^k(y) < p \\ 1  \quad \text{if} \;\;\; f_{(B,p, +)}^k(y) \geq p.
\end{cases} \end{equation*}
In dynamical systems terminology, $\tau_{(B,p,-)}(y)$ and $\tau_{(B,p,+)}(y)$ are called the \linebreak
{\bf itineraries} of the point $y$.   
\end{definition}

\noindent {\bf Algorithm}
\m

\noindent Input: An admissible pair of strings $(\alpha, \beta)$ and an $x\in \R^+$.
\m

\hangindent=5mm  \hangafter 1
\noindent Output:  The decimals representations ${\sigma}$ and $\omega$ of $x$ in the binary radix systems  $R^{\bullet}_{(\alpha, \beta,-)}$ and $R^{\bullet}_{(\alpha, \beta,+)}$, respectively. \B

1.  Find the least $x\in (0,1)$ such that  $\sum_{n=0}^{\infty} \alpha_n \, x^n = 
\sum_{n=0}^{\infty} \beta_n \, x^n$; call it $b$ and let $B=1/b$.\m

2. Compute  $p:= \pi(\alpha) = \pi(\beta)$, where $\pi\, : \, \Omega \rightarrow [0,1]$ is the radix map.  \m

3.  Find the minimum non-negative integer $N$ such that $y:= b^{N}(1-b) x < p$. \vskip 2mm

4. Compute $\tau_{(B,p,-)}(y) =  \sigma_0\, \sigma_1 \, \sigma_2 \cdots$ and 
$\tau_{(B,p,+)}(y) =  \omega_0\, \omega_1 \, \omega_2 \cdots$. 
\m

4. Return $$\begin{aligned} { \sigma} &=  \sigma_0\, \sigma_1 \, \sigma_2 \cdots \sigma_{N} {\bb} \,\sigma_{N +1}\, \sigma_{N+2} \cdots,\quad
\text{and} \\
 {\omega} &=  \omega_0\, \omega_1 \, \omega_2 \cdots \omega_{N} {\bb} \, \omega_{N +1}\, \omega_{N+2} \cdots .
\end{aligned}$$ \m

\begin{definition} \label{def:section}  Denote the output $\sigma$ and $\omega$ of the algorithm by   $\tau_{(\alpha,\beta,-)}(x)$ and  $\tau_{(\alpha,\beta,+)}(x)$, respectively, and call  $\tau_{(\alpha,\beta,\pm)} \, : \, \R^+ \rightarrow \Omega^{\bullet}_{(\alpha, \beta,\pm)}$ the {\bf section maps}. In other words, these give the
decimal representions of $x$ in $R_{(\alpha,\beta,\pm)}$.  Note that $\tau_{(\alpha,\beta,\pm)}$ 
takes values in  $\Omega^{\bullet}_{(\alpha, \beta,\pm)}$ while $\tau_{(B,p,\pm)}$ takes values in  $\Omega$.
\end{definition}

\begin{example}[Golden Ratio Radix System]    This is a continuation of  Examples~\ref{ex:2},\;\ref{ex:22},\; \ref{ex:222},  and \ref{ex:2222} on the golden ratio radix system.  The relavant admissible pair is $(\overline{01},1\overline 0)$.  In this case $B = \frac {1+\sqrt{5}}{2}$ and  $p = \frac{3-\sqrt{5}}{2}$.  
If, for example,  $x = 2$, then $N=2$ and
$$\tau_{(\alpha,\beta,-)}(2) = \tau_{(\alpha,\beta,+)}(2) = 10_{\bb}0\overline{01}.$$
\end{example}

The following theorem suffices to prove the validity of the algorithm. 

\begin{theorem}  \label{thm:inverse}
The section map $\tau_{(\alpha,\beta,\pm)} \, : \, \R^+ \rightarrow \Omega^{\bullet}_{(\alpha, \beta,\pm)}$ is the inverse of the radix map $\overset{\bullet}{\pi} \, : \,\Omega^{\bullet}_{(\alpha, \beta,\pm)} \rightarrow \R^+$.  
\end{theorem}

\begin{proof}   We will prove that  $\tau_{(\alpha,\beta,-)}$ is the inverse of $\pi$;
a similar proof holds for  $\tau_{(\alpha,\beta,+)}$.
We first show that the functions $f_{(B,p, \pm)}$ are well defined, i.e. $1< B \leq 2$ and 
 $1-b \leq p \leq b$.  Clearly $1< B \leq 2$ because $1 > b \geq 1/2$ as stated in Theorem~\ref{thm:BV}.  
Moreover  $1-b \leq p \leq b$ because
$$1-b \leq (1-b) \sum_{n=0}^{\infty} \beta_n b^n   = p =  (1-b)\sum_{n=0}^{\infty} \alpha_n b^n \leq 
  (1-b)\sum_{n=1}^{\infty}  b^n = b.$$

To simplify notation, abbreviate $\tau := \tau_{(\alpha,\beta,-)}$ and  ${\widehat \tau} := \tau_{( B,p,-)}$. We next show that $\pi\circ {\widehat \tau}$ is the identity on $[0,1]$.  If $f_0(x) = Bx, \; f_1(x)  = Bx + (1-B)$, and  $g_0(x) = bx, \; g_1(x)  = bx + (1-b)$, then  $f_0$ and $g_0$ are inverses, as are $f_1$ and $g_1$. 
Expressing $g_{i}(x)  = bx+i(1-b)$ for $i=0,1$ and iterating
$$g_{\omega_0} \circ g_{\omega_1} \circ g_{\omega_{2}} \circ \cdots \circ g_{\omega_k}(x) = 
b^k x + (b^{k-1} \omega_{k-1} + \cdots + a\omega_1 + \omega_0)(1-b).$$
Therefore, for any $x_0$, we have
$$\lim_{k\rightarrow \infty} g_{\omega_0} \circ  g_{\omega_1} \circ  g_{\omega_2} \circ \cdots\circ
 g_{\omega_k} (x_0) = (1-b) \sum_{k=0}^{\infty} \omega_k \, b^k = \pi(\omega) .$$
 Let $M_0 = [0,p]$ and $M_1= (p,1]$.  Let $x\in [0,1]$ and ${\widehat \tau}(x) = \omega$.  It follows from the definition of ${\widehat \tau}$ that 
$$x \in M_{\omega_0}, \quad f_{\omega_0}(x) \in M_{\omega_1}, \quad f_{\omega_1}\circ f_{\omega_0}(x) \in  M_{\omega_2},  \quad  f_{\omega_2}\circ f_{\omega_1}\circ  f_{\omega_0}\in M_{\omega_3},\dots$$ and therefore 
$$x \in  g_{\omega_0}( M_{\omega_1}),\quad  x\in  g_{\omega_0}\circ g_{\omega_1} (M_{\omega_2}), \quad
 g_{\omega_0}\circ g_{\omega_1}\circ g_{\omega_2} (M_{\omega_3}), \dots.$$  Hence
$(\pi\circ {\widehat \tau})(x) =  \lim_{k\rightarrow \infty} g_{\omega_0} \circ  g_{\omega_1} \circ  g_{\omega_2} \circ \cdots\circ  g_{\omega_k} (x_0) = x.$ 

To show that  $\tau$ is the inverse of $\overset{\bullet}{\pi}$, with notation as in the algorithm and  letting $\sigma := \tau(x)$, we have
$$(\overset{\bullet}{\pi} \circ \tau)(x) = \overset{\bullet}{\pi} (\sigma) = \frac{B^N}{1-b} \, \pi({\widehat \sigma}) =
 \frac{B^N}{1-b} \, (\pi\circ {\widehat \tau}) (y) =  \frac{B^N}{1-b} \,  b^N(1-b) x = x.$$

It remains to show that the image of any $x\in \R^+$ under the map $\tau$ lies in $\Omega^{\bullet}_{(\alpha, \beta,-)}$.   With $y$ as defined in the algorithm, it is 
 sufficient to show that ${\widehat \tau}(y) \in  \Omega^0_{(\alpha, \beta,-)}$. 
Let $\tau_-$ and $\tau_+$ denote the itineraries of the point $p$ 
of the functions  $f_{(B,p, -)}$ and $f_{(B,p, +)}$,  respectively.  In \cite[Theorem 5.1]{BHV} it is proved
that ${\widehat \tau}([0,1]) =  \{\omega\in \Omega \, : \, S^n\omega \notin (\tau_-,\tau_+] \; \text{for all} \; n\geq 0\}$, and in \cite[Theorem 1.1]{BV} it is proved that $\alpha = \tau_-$ and $\beta=\tau_+$.  Therefore  ${\widehat \tau}([0,1]) = \Omega_{(\alpha, \beta,-)}$.  It only remains to show that if  
${\widehat \omega} = {\widehat \tau} (y)$,
 then $0\widehat \omega \preceq \alpha$.  However, since  $y<p$  as in step 1 of the algorithm,  it follows immediatly from the definition of  the itinerary ${\widehat \tau}(y)$ of $y$ that ${\widehat \omega}_0 = 0$.  In particular $0{\widehat \omega} \prec \alpha$.
\end{proof}

\begin{cor}  For each base $B, \; 1 < B < 2$, there exist infinitely many binary radix systems with base $B$.   
\end{cor}

\begin{proof}  Given $B$, let $p$ be any real number such that $1-1/B \leq p \leq 1/B$, and consider the two 
functions $f_{-}$ and $f_{+}$ defined in Equations~(\ref{eq:f-}) and (\ref{eq:f+}), respectively.  
If $\alpha_p$ is the itinerary of the point $p$ for the function $f_{-}$, and $\beta_p$ is the itinerary of the point $p$ for the 
function $f_{+}$, then $(\alpha_p,\beta_p)$ is an admissible pair.  The proof of this fact appears in
\cite[Theorem 4.7]{BV}.  Therefore $R^{\bullet}_{(\alpha_p,\beta_p,-)}$ and $R^{\bullet}_{(\alpha_p,\beta_p,-)}$ are binary radix systems.  

For each $B$, however, there are infintely many choices for $p$.  Each choice of $p$ leads to a distinct
admissible pair $(\alpha_p,\beta_p)$ because the maps $p \mapsto \alpha_p$ and $p \mapsto \beta$ are increasing as a
function of $p \in [1- 1/B, 1/B]$.  To verifty that $p \mapsto \alpha_p$ is increasing (the proof for $p \mapsto \beta$ is similar), let $f_p = f_{(B,p,-)}$ and $\tau_p = \tau_{(B,p,-)}$.   Assume  $p' > p$ and let $x_n = f_{p}^n(p), \, x'_n = f_{p'}^n(p')$ and let $\alpha = \alpha_p, \, \alpha' = \alpha_{p'}$.   Note that $\alpha_p \neq \alpha_{p'}$; otherwise $|x'_n - x_n| = B^n|p'-p|$ for all $n$, which is not possible because $B>1$.  Therefore assume that $\alpha_k = \alpha'_k$ for $0\leq k\leq n-1$, but $\alpha_n \neq \alpha'_n$.  By elementary analytic geometry, if $\alpha_k = \alpha'_k = 1$, then 
$x_k - x_{k+1} > x'_n- x'_{k+1}>0$, and if   $\alpha_k = \alpha'_k = 0$, then $x'_{k+1} -x'_k >  x_{k+1} -  x_{k}>0$
for any $k$.  From this it is easy to deduce that if $\alpha_{n-1} = \alpha'_{n-1} = 1$,  then $\alpha_n = 0$ and $\alpha'_n = 1$ and hence $\alpha < \alpha'$, and if  $\alpha_{n-1} = \alpha'_{n-1} = 0$, then $\alpha_n = 0$ and $\alpha'_n = 1$ and hagain  $\alpha < \alpha'$.
\end{proof}

\section{Proof of Theorem~\ref{thm:converseR}} \label{sec:proofs}

This section contains the proof of Theorem~\ref{thm:converseR}.

\begin{proof} (Theorem~\ref{thm:converseR})    For $\Gamma \subseteq  \Omega^{\bullet}$, denote
$\widehat \Gamma = \{{\widehat\omega}\,: \, \omega \in \Gamma\}$.  Let 
$$\begin{aligned} \alpha &= \sup \, \{ \gamma \in {\widehat \Gamma} \, :  \, \gamma_0 = 0\}, \\ 
\beta &= \inf \, \{ \gamma  \in \widehat{\Gamma} \, : \, \gamma_0 = 1\}. \end{aligned}$$  
We first show that the pair $(\alpha,\beta)$ satisfies conditions (1) and (2) in Definition~\ref{def:allowable} of an allowable pair.  
It follows readily  from the definition of $\alpha$ and $\beta$ and from the shift invariance of $\Gamma$ that
\begin{equation*} \label{eq:1} 
S^n\alpha \notin (\alpha,\beta)\quad \text{and} \quad S^n\beta \notin (\alpha,\beta)\end{equation*}
for all $n\geq 0$, which is close to, but not quite, condition (2) in Definition~\ref{def:allowable} of an allowable pair.  Condition (1) in  Definition~\ref{def:allowable} holds because, by the shift invariance of $\Gamma$, 
 there are strings in $\Gamma_0$ that begin with $01$ and stings in $\Gamma_1$ that begin with $10$.  

To show that $(\alpha,\beta)$ satisfies condition (2) in Definition~\ref{def:allowable}, it only remains to prove that there is no $n\geq 0$ such that  $S^n\alpha = \beta$ and no $n\geq 0$ such that  $S^n\beta = \alpha$.
The fact that there is no element of $\widehat \Gamma$ between $\alpha$ and $\beta$ in the lexicographic order and that $\overset{\bullet}{\pi}c$ is increasing and surjective forces $\overset{\bullet}{\pi}(\bb \alpha) = \overset{\bullet}{\pi}(\bb \beta)$.  Moreover, either $\bb \alpha \in \Gamma$ or $\bb \beta \in \Gamma$, but not both.  We will assume that $\bb\alpha \in \Gamma$ and $\bb \beta \notin \Gamma$; the proof in the case that $\bb \beta \in \Gamma$ is essentially the same.  There is no $n\geq 0$ such that $S^n\alpha = \beta$; otherwise the fact that $\alpha \in {\widehat \Gamma}$ and the shift invariance of $\Gamma$ (and hence the shift invariance of $\widehat \Gamma$) would imply that $\beta \in {\widehat \Gamma}$, a contradiction.  
 Finally assume, by way of contradiction, that there is an $n\geq 0$ such that 
$S^n\beta = \alpha$.  Then $\beta = t \alpha \prec t\beta$, where $t$ is a finite string.  If there exists a $\gamma \in {\widehat \Gamma}$ such that $t\alpha = \beta \prec \gamma \prec t\beta$, then $\alpha = S^n(t\alpha) \prec S^n\gamma \prec S^n(t\beta) = \beta$.  But by the definition of $\alpha$ and $\beta$ as $\sup$ and $\inf$, there can be no such $\gamma \in  {\widehat \Gamma}$.  Therefore there is no such $\gamma$ with  $\beta \prec \gamma \prec t\beta$, which contradicts the definition of $\beta$ as $\inf \, \{ \gamma  \in \widehat{\Gamma} \, : \, \gamma_0 = 1\}$ since
$\beta \notin {\widehat \Gamma}$.

 The next claim is that  $\Gamma =  \Omega^{\bullet}_{(\alpha, \beta,-)}$.
The definition of $\alpha$ and $\beta$ implies that ${\widehat \Gamma}\subseteq \Omega_{(\alpha,\beta,-)}$. That $\Gamma$ is shift invariant further implies that ${\widehat \Gamma}\subseteq \Omega^0_{(\alpha,\beta,-)}$, and therefore that $\Gamma \subseteq \Omega^{\bullet}_{(\alpha,\beta,-)}$. To prove equality we take a somewhat circuitious route. 
Let $b = 1/B$ and define $\pi_b \, : \, \Omega \rightarrow [0,1]$ by
$$\pi_b (\omega):= (1-b) \, \sum_{n=0}^{\infty} \omega_ n \, b^n,$$
which is just the radix map in Equation~\eqref{eq:radixM1} except defined on all of $\Omega$.  
Define 
$$\Gamma' := \{ t \gamma \, :\, \gamma \in {\widehat \Gamma}, \, t \, \text{a finite string of} \, 1's \, \text{including the empty string}\}.$$
We claim that  $\pi_b \, :\,\Gamma' \rightarrow [0,1]$ is strictly increasing and surjective.  To see that
it is strictly increasing, note that, since $\overset{\bb}{\pi}$ is  strictly increasing on $\Gamma$,
the map $\pi_b \;: \,  {\widehat \Gamma} \rightarrow [0,q]$ is stringly increasing and surjective for some $0<q\leq 1$.  Moreover, if $\omega \in \Gamma'\setminus {\widehat \Gamma}$ and $\gamma \in {\widehat \Gamma}$, 
then, by Definition~\ref{def:address} of the address spaces (recall that $\Gamma \subseteq \Omega^{\bullet}_{(\alpha,\beta,-)}$), we have $0\gamma \preceq \alpha \prec 0\omega$.  Therefore
$\gamma \prec \omega$, i.e., every element of ${\widehat \Gamma}$ is less than every element of $\Gamma'\setminus {\widehat \Gamma}$. 
In addition, since  $0\gamma \preceq \alpha$, we have $\pi_b(\gamma) \leq (1/b)\, \pi_b(\alpha)$ by the fact that $\pi$ is stringly increasing on $\Gamma$.  And similarly, since $\alpha \prec 0\omega$, we have
$(1/b)\, \pi_b(\alpha) < \omega$.  Therefore $\pi_b(\gamma) < \pi_b(\omega)$.   Finally, if $\sigma$ and $\omega$ both lie in $\Gamma'\setminus {\widehat \Gamma}$ and
$\sigma \prec \omega$, we will show that $\pi_b(\sigma) < \pi_b(\omega)$. Let $\sigma = t_1 \gamma_1\in \Gamma'\setminus {\widehat \Gamma}$, where
$t_1$ is a string of $m$ ones and $\gamma_1\in {\widehat\Gamma}$, and $\omega = t_2  \gamma_2 \in \Gamma'\setminus {\widehat \Gamma}$, where $t_2$ is a string of $n\geq m$ ones and $\gamma_2\in {\widehat\Gamma}$. Then $t\gamma_2 \succ \gamma_1$, where $t$ is a string of $n-m$ ones.  Therefore $\pi_b(t\gamma_2) \succ \pi_b(\gamma_1)$ and hence $\pi_b(\omega) \succ \pi_b(\sigma)$.  Thus $\pi_b$ is  strictly increasing on $\Gamma'$.  

To show that $\pi_b \, :\,\Gamma' \rightarrow [0,1]$ is surjective, we express $\Gamma'$ as the union of non-overlapping intervals and show that the images of these intervals under $\pi_b$ leave no gaps. 
First note that the greatest element of $\widehat \Gamma$ is $S\alpha$; the smallest element of $1{\widehat \Gamma} \setminus {\widehat \Gamma}$ is $1S^2\alpha$ and the largest is $1S\alpha$;  the smallest element of $11{\widehat \Gamma} \setminus ({\widehat \Gamma} \cup 1{\widehat \Gamma})$ is  $11S^2\alpha$ and the largest is $11S\alpha$; etc.  However, $\pi_b(1S^2\alpha) = \pi_b(S\alpha)$, and therefore  $\pi_b(11S^2\alpha) = \pi_b(1S\alpha)$, etc.  Hence $\pi_b \, :\,\Gamma' \rightarrow [0,1]$ is surjective.  

To conclude the proof that $\Gamma =  \Omega^{\bullet}_{(\alpha, \beta,-)}$, call a map $\tau \, :\, [0,1] \rightarrow \Omega$ such that $\pi_b\circ \tau$ is the identify a {\it section} of $\pi_b \, :\Omega \rightarrow [0,1]$.  If $\tau([0,1])$ is shift invariant, then $\tau$ is called a {\it shift invariant section}.  In our case, the inverse of $\pi_b$ restricted to $\Gamma'$ is a shift invariant section.  Call this section $\tau_b$.  A {\it mask} (in our case) is a partition of the interval $[0,1]$ into two parts $M_0$ and $M_1$.  Given a mask $M = \{M_0,M_1\}$, define a function
$$f_{(B,M)} = \begin{cases}  Bx \quad & \text{if} \; x\in M_0 \\ BX+(1-B) \quad & \text{if} \; x\in M_1.$$
\end{cases}$$ The {\it itinerary map} $\tau_{(B,M)} \, : \, [0,1] \rightarrow \Omega$ is defined by
\begin{equation*}
[\tau_{(B,M}(x) ]_k = \begin{cases} 0 \quad \text{if} \;\;\; f_{(B,M)}^k(x)  \in M_0 \\ 1  \quad \text{if} \; \;\;f_{(B,M)}^k(x) \in M_1.  
\end{cases} \end{equation*}
 To continue the proof, we use  \cite[Theorem 4]{BV2} which states the following:  
the itinerary map $\tau_{(B,M)}$ is a  section of $\pi_b$, and conversely, every shift invariant section of $\pi_b$ is of the above form.  Hence $\tau_b  = \tau_{(B,M)}$ for some mask $M$.  Because $\tau_b$ is increasing and
because  $0\sigma \prec 1\omega$ for any $\sigma,\omega \in \Omega$, the mask must be of the form $M_0 = [0,p], \, M_1 = (p,1]$ for some $p\in (0,1)$.  In particular, $\tau_b= {\widehat \tau}_{(B,p,-)}$ as in Definition~\ref{def:itinerary} for some $p\in (0,1)$.
Let $\tau_-$ and $\tau_+$ denote the itineraries ${\widehat \tau}_{(B,p,-)}(p)$ and ${\widehat \tau}_{(B,p,+)}(p)$, respectively, of the point $p$.  Then  
$\Omega_{(\tau_-,\tau_+,-)} =   \tau_b([0,1]) =  \Gamma' \subseteq \Omega_{(\alpha,\beta,-)}$,  the first equality by \cite[Theorem 5.1]{BHV}.  If  $\Gamma'  \neq \Omega_{(\alpha,\beta,-)}$, then either $\tau_- \prec \alpha$ or $\tau_+ \succ \beta$.  In either case there is a contradiction to the definition of $\alpha$ as a sup or $\beta$ as an inf.  Therefore
 $\alpha = \tau_-, \, \beta = \tau_+$ and $\Gamma'  = \Omega_{(\alpha,\beta,-)}$, which implies that  ${\widehat \Gamma}  = \Omega^0_{(\alpha,\beta,-)}$, which in turn implies that $\Gamma = \Omega^{\bullet}_{(\alpha,\beta,-)}$.

To conclude thet proof that $(\alpha,\beta)$ is allowable,  we must prove condition (3) in  Definition~\ref{def:allowable}.   A result of Parry \cite[Page 373]{P} on the topological entropy of a dynamical system on the unit interval, where the function is of the form $f_{(B,p,\pm)}$ shown in Figure 1,  provides the first equality in 
$$0 < \log (B) = h(\tau_b([0,1])) = h(\Omega_{(\tau_-,\tau_+,-)}) = h(\Omega_{(\alpha,\beta,-)} ).$$
The first inequality is because $B>1$.  Since $(\alpha,\beta)$ is the pair of critical itineraries $(\tau_-,\tau_+)$ of functions of the form in Equations~\eqref{eq:f-} and \eqref{eq:f+}, and since the critical itineraries of such a function
are an admissible pair \cite[Theorem 4.7]{BV}, we now know that $(\alpha,\beta)$ is an admissible pair.  

To prove that $(\Gamma, B) = R^{\bullet}_{(\alpha,\beta,-)}$ it must be shown that
\begin{enumerate}
\item  $\Gamma =  \Omega^{\bullet}_{(\alpha, \beta,-)}$, and
\item $B =  B_{(\alpha,\beta)} := 1/r_{(\alpha,\beta)}$.
\end{enumerate}
We have already proved (1).   Concerning (2), it was part of the proof of Theorem~\ref{thm:inverse} that $\Omega_{(\alpha,\beta,-)} = \tau_{(B_{(\alpha,\beta)},\pi(\alpha),-)}([0,1])$, and it was shown above
that  $\Omega_{(\tau_-,\tau_+,-)} =  {\widehat \tau}_{(B,p,-)}$.  Again using the result  \cite[Page 373]{P}
we have
$$\begin{aligned}  \log (B) &=  h({\widehat \tau}_{(B,p,-)}([0,1])) =  h(\Omega_{(\tau_-,\tau_+,-)}) = h(\Omega_{(\alpha,\beta,-)} ) \\ &=
h( \tau_{(B_{(\alpha,\beta)},\pi(\alpha),-)}([0,1]) )= \log ( B_{(\alpha,\beta)}). \end{aligned}$$ 
\end{proof}

\section{Radix Tilings of the Real Line} \label{sec:tiling}

For an element $\omega \in \Omega^{\bullet}$,
let $\omega_{\bullet}$ denote the finite substring of $\omega$ to the left of the decimal point. 

\begin{definition}  Given a  binary radix system $(\Gamma, B)$  with radix map $\overset{\bullet}{\pi}$ and a finite string $s$, let 
$$T'_{s} := \{ \overset{\bullet}{\pi}(\omega) \, :  \, \omega \in \Gamma \; \; \text{and} \;\; \omega \bb= s\}.$$
Note that, for many values of $s$, the set $T'_{s}$ may be empty.  For instance, in Example~\ref{ex:2} the set $T'_{011} = \emptyset$.  If $T'_s \neq \emptyset$, then the closure $T_s$ of $T'_s$ is a closed interval which we call a {\bf tile}.  If $\Omega_F$ denotes the set of all finite binary strings, let
$${\cal T} := \{ T_s \, : \, s \in \Omega_F \}.$$
Then $\cal T$ is a collection of non-overlapping intervals whose union is $\R^+$. The set  $\cal T$ will be referred to as the {\bf tiling} of $\R^+$ for the binary radix system $(\Gamma,B)$.  For an $(\alpha,\beta)$-radix system, the corresponding
tiling is denoted by ${\cal T}_{(\alpha, \beta)}$.  The tiling for $R^{\bullet}_{(\alpha,\beta,-)}$ is the same as for $R^{\bullet}_{(\alpha,\beta,+)}$.
\end{definition}

\begin{example}[{\it Standard Radix System}]  \label{ex:two} For the standard binary radix system (Examples~\ref{ex:1},\;\ref{ex:11}, and \ref{ex:111}), the tiling is the set ${\cal T} = \{  [n,n+1] \, : \,\ n\geq 0 \}$ of unit length intervals. 
\end{example}

\begin{example}[{\it Golden Ratio Radix System}] \label{ex:golden}
 For the golden ratio based radix system (Examples~\ref{ex:2},\;\ref{ex:22}, and \ref{ex:222}), there are tiles of two lengths in the ratio $1\, : \, \frac{1+\sqrt{5}}{2}$.  There are tiles of length $1/\tau$ whose ``fractional part" (the part to the right of the decimal point) ranges from $_{\bb}000\cdots$ to $_{\bb}010101\dots$ and tiles of length $1$ whose ``fractional part"  ranges from $_{\bb}000\cdots$ to $_{\bb}101010\dots$,
The tiling $\cal T$ is a well known non-periodic tiling of the line.   If the tiles are denoted $1$ and $B$ (for their relative lengths), then the sequence of tiles in the tiling of $\R^+$, from left to right, is
\begin{equation} \label{B1}   B\, 1 \,B \,B\,1 \,B \,1 \,B \,B \,1 \,B \,B \,1 \,B \,1 \,B \,B \,1 \, B \,1 \,B \,B \,1 \cdots.
\end{equation}
This tiling is self-replicating in the following sense.  If $f_B \, : \, \R^+\rightarrow \R^+$ is the function
$f_B(x) = Bx$, then $B({\cal T}) := \{ f_B(T) \, : \, T \in {\cal T} \}$ is also a tiling of $\R^+$, and each tile
in   $B({\cal T})$ is the union of tiles in $\cal T$.  More specifically, each tile of type $B$ is an interval of 
the form $T_B  := [s0_{\bb}\overline{0}, s0_{\bb}\overline{10}]$ for some finite string $s$, and $B(T_b) =  [s00_{\bb}\overline{0}, s01_{\bb}\overline{01}] =  [s00_{\bb}\overline{0}, s00_{\bb}\overline{10}] \cup
 [s01_{\bb}\overline{0}, s01_{\bb}\overline{01}]$, which is the union of a tile of type $B$ and a tile of type $1$.
Likewise each tile of type $1$  is an interval of  the form $T_B  := [s_{\bb}\overline{0}, s_{\bb}\overline{01}]$ for some finite string $s$, and $B(T_1) =  [s0_{\bb}\overline{0}, s0_{\bb}\overline{10}]$ which is a tile of type
$B$. Therefore, the sequence (\ref{B1}) above can be recursively generated, 
starting from $B$ and using the substitution rules 
$$\begin{aligned}  B &\leftarrow B1 \\ 1 &\leftarrow B. \end{aligned}$$  In other words, the tiling is
recursively generated as follows:
$$B \rightarrow B1 \rightarrow B1B \rightarrow B1BB1 \rightarrow B1BB1B1B \rightarrow\cdots.$$   
\end{example}

\begin{definition} With notation as in the example above,  call a tiling ${\cal T}$ {\bf self-replicating} if each tile
in $B({\cal T})$ is the union of tiles in $\cal T$.  
\end{definition}

\begin{theorem} If $(\alpha,\beta)$ is an admissible pair, then 
\begin{enumerate}
\item the tiling  ${\cal T}_{(\alpha,\beta)}$ is self-replicating, and
\item if $\alpha$ and $\beta$ are eventually periodic, then there are at most  finitely many
lengths of tiles in the tiling ${\cal T}_{(\alpha,\beta)}$.
\end{enumerate}
\end{theorem}

\begin{proof} Let $(\Gamma, B)$ be the binary radix system with radix map $\overset{\bullet}{\pi}$.  Concerning statement (1), consider the set $X = \{x_1, x_2, \dots \}$ of right endpoints of the tiles in 
 ${\cal T}_{(\alpha,\beta)}$. It suffices to prove that $Bx_n \in X$ for all $n\geq 1$.   
For any $n\geq 1$, the decimal representation of $x_n$ must be of the form 
\begin{equation*} \label{eq:R} \tau_{(\alpha,\beta,-)}(y) =  v \alpha_0 \alpha_1 \cdots \alpha_{N} {\bb} \alpha_{N+1}\alpha_{N+2} \cdots\end{equation*} for some finite string $v$ and some $N \geq 0$. 
 To see this, suppose that $x_n$ is the right endpoint of the tile $T = \{ \overset{\bullet}{\pi}(\omega)\, :\, \omega_{\bullet} = s\}$ and note that $s$ has the form  $s = s' \alpha_0 \alpha_1 \cdots  \alpha_N$ for some finite string $s'$ and some $N$.  Therefore any point in $T$ has a decimal representation of the form 
 $$s' \alpha_0 \alpha_1 \cdots \alpha_N{\bb}\omega_1\omega_2 \cdots \preceq s' \alpha_0 \alpha_1 \cdots \alpha_N{\bb}\alpha_{N+1} \alpha_{N+2} \cdots \in \Omega^{\bullet}_{(\alpha,\beta)}$$ for some string $\omega_1\omega_2 \cdots $.  Conversely, any point with decimal representation \linebreak
 $v \alpha_0 \alpha_1 \cdots \alpha_{N} {\bb} \alpha_{N+1}\alpha_{N+2} \cdots \in \Omega^{\bullet}_{(\alpha,\beta)}$
is a right enpoint of some tile in  ${\cal T}_{(\alpha,\beta)}$ because any decimal $v \alpha_0 \alpha_1 \cdots \alpha_N{\bb}\omega_1\omega_2 \cdots \in  \Omega^{\bullet}_{(\alpha,\beta)}$ is less than
$v \alpha_0 \alpha_1 \cdots \alpha_{N} {\bb} \alpha_{N+1}\alpha_{N+2} \cdots$ 
  in the lexicographic order.  Therefore $Bx_n \in X$ since the decimal representation of $Bx_n$ is obtained by
shifting the decimal point one space:
 $v \alpha_0 \alpha_1 \cdots \alpha_{N}  \alpha_{N+1}{\bb}\alpha_{N+2} \cdots$. 

Concerning statement (2), from the form of the decimal representation of the endpoints of the intervals of the tiling (as
discussed above), the lengths of the tileshave values $\overset{\bullet}{\pi} (\bb S^m \alpha) -\overset{\bullet}{\pi}(\bb S^n\beta)$ for some $n,m\geq 0$.  Since $\alpha$ and $\beta$ are eventually periodic, there are at most finitely many shifted strings $S^m\alpha$ and $S^n\beta$. 
\end{proof}

\begin{example} This is a continuation of Example~\ref{ex:3} in which three distinct $(\alpha,\beta)$-radix systems have the same base $B$ such that  $B^3-B-1 = 0$.  The three associated self-replicating tilings, however, are distinct.  From the self-replicating property, in a similar way as was as was done in Example~\ref{ex:golden}, substitution rules for the three tilings can be determined.  

In the notation of  Example~\ref{ex:3}, for the tiling ${\cal T}_{(\alpha_1,\beta_1)}$ there are five tile types $T_0,T_1,T_2,T_3,T_4$ of relative lengths
$1, B ,B^2, B^3, B^4$, respectively.  To simplify notation we refer to tile $T_i$ as simply $i$.  Then the  substitution rules are:
$$  0 \leftarrow 1  \qquad \qquad  1 \leftarrow 2   \qquad \qquad 2 \leftarrow 3     \qquad \qquad  3 \leftarrow 4     \qquad \qquad 4 \leftarrow 40,$$ 
and the tiling begins, from left to right:
$$4\, 0\, 1\,2\,3\,4\,4\,0\,4\,0\,1\,4\,0\,1\,2\,4\,0\,1\,2\,3\,4\,\,0\,1\,2\,3\,4\,4\,0\,1\,2\,3\,4\,4\,0 \dots.$$
For the tiling ${\cal T}_{(\alpha_2,\beta_2)}$ there are 4 tile types $T_0,T_1,T_2,T_3$ of relative lengths
\linebreak
$1, B ,B^2, B^6$, respectively.   The  substitution rules are:
$$  0 \leftarrow 1  \qquad \qquad  1 \leftarrow 2   \qquad \qquad 2 \leftarrow 10     \qquad \qquad  3 \leftarrow 32,$$ 
and the tiling begins, from left to right:
$$3\, 2\, 1\,0\,2\,1\,1\,0\,2\,2\,1\,1\,0\,1\,0\,2\,2\,1\,2\,1\,1\,\,0\,1\,0\,2\dots.$$
For the tiling ${\cal T}_{(\alpha_3,\beta_3)}$ there are 4 tile types $T_0,T_1,T_2,T_3$ of relative lengths \linebreak
$1, B ,B^2, B^5$, respectively.   The  substitution rules are:
$$  0 \leftarrow 1  \qquad \qquad  1 \leftarrow 2   \qquad \qquad 2 \leftarrow 01     \qquad \qquad  3 \leftarrow 31,$$ 
and the tiling begins, from left to right:
$$3\, 1\, 2\,0\,1\,1\,2\,2\,0\,1\,0\,1\,1\,2\,1\, 2\,2\,0\,1\,2\,0\,\,1\,0\,1\,1\,2\dots.$$
\end{example}

All three of the above tilings are self-replicating under expansion by the factor $B$.

\end{document}